\documentclass[12pt]{amsart}
\usepackage{amssymb}
\usepackage{amsfonts}
\usepackage{graphicx}
\usepackage{amsmath}
\usepackage{xcolor}

\newtheorem{theorem}{Theorem}

\newtheorem{lemma}[theorem]{Lemma}

\newtheorem{proposition}[theorem]{Proposition}

\begin{document}
\title[Uniqueness results ]{Fischer decompositions for
entire functions of sufficiently low order}
\author{J.M. Aldaz and H. Render}
\address{H. Render: School of Mathematical Sciences, University College
	Dublin, Dublin 4, Ireland.}
\email{hermann.render@ucd.ie}
\address{J.M. Aldaz: Instituto de Ciencias Matem\'aticas (CSIC-UAM-UC3M-UCM)
	and Departamento de Matem\'aticas, Universidad Aut\'onoma de Madrid,
	Cantoblanco 28049, Madrid, Spain.}
\email{jesus.munarriz@uam.es}
\email{jesus.munarriz@icmat.es}
\thanks{The first named author was partially supported by Grant PID2019-106870GB-I00 of the
	MICINN of Spain,  by ICMAT Severo Ochoa project 
	CEX2019-000904-S (MICINN), and by  V PRICIT (Comunidad de Madrid - Spain).}

\maketitle

\begin{abstract} The existence of decompositions of the form $
	f=P\cdot q+r$ with $P_k^{\ast}\left(  D\right)  r=0
$, where $f$ is entire, $P$ a polynomial and  $P^{\ast}_k$ the principal part of $P$ with its coefficients conjugated, was achieved in \cite{AlRe23} under certain restrictions on the order of $f$. Here we prove uniqueness, thereby obtaining Fischer decompositions, under conditions that sometimes match those required for existence, and sometimes are more restrictive, depending on the parameters involved. 
\end{abstract}

\section{Introduction}

Let $\mathcal{P}\left(  \mathbb{C}^{d}\right)  = \mathbb{C}[z_{1}, \dots,
z_{d}]$ be the set of all polynomials in $z=\left(  z_{1},\dots,z_{d}\right)
\in\mathbb{C}^{d}$, and let $P \in\mathcal{P} \left(  \mathbb{C}^{d}\right)
$. By $P^{\ast}\left(  z\right)  $ we denote the polynomial obtained from
$P\left(  z\right)  $ by conjugating its coefficients, and by $P\left(
D\right)  $, the linear differential operator obtained by replacing each
appearance of the variable $z_{i}$ with $\partial/\partial{z_{i}}$, for
$i=1,\dots,d$. Fischer's decomposition theorem states that for every
\emph{homogeneous} polynomial $P$ and every polynomial $f$ there exist
\emph{unique} \emph{polynomials} $q$ and $r$ such that%
\[
f=P\cdot q+r\text{ and }P^{\ast}\left(  D\right)  r=0.
\]
According to D. J. Newman and H. S. Shapiro, cf. \cite[pg. 971]{NeSh66},
Fischer's result ``underlies various formal schemes for exhibiting a basic set
of polynomial solutions of a partial differential equation". H. S. Shapiro and
several of his collaborators studied Fischer decompositions in a wider setting
(cf. for instance \cite{Shap89}, \cite{KhSh92}, \cite{EbSh95}, \cite{EbSh96})
going beyond the case of polynomials to more general function spaces, in
particular, to the space $E\left(  \mathbb{C}^{d}\right)  $ of all entire
functions $f:\mathbb{C}^{d}\rightarrow\mathbb{C}$. We recall a notion
introduced in \cite[p. 522]{Shap89}: given a vector space $E$ of infinitely
differentiable functions $f:G\rightarrow\mathbb{C }$ (defined on an open
subset $G$ of $\mathbb{C}^{d}$) that is a module over $\mathcal{P}\left(
\mathbb{C}^{d}\right)  $, a polynomial $P$ and a differential operator
$Q\left(  D\right)  $ are said to form a \emph{Fischer pair for the space }%
$E$, if for each $f\in E$ there exist \emph{unique} elements $q\in E$ and
$r\in E$ such that
\begin{equation}
	f=P\cdot q+r\text{ and }Q\left(  D\right)  r=0. \label{eqDecomp}%
\end{equation}
 The question of uniqueness can be rephrased as follows: suppose that $f=P\cdot
q_{j}+r_{j}$ and $Q\left(  D\right)  r_{j}=0$ for $j=1,2.$ Then
with $q:=q_{2}-q_{1}$ and $r := r_{2}-r_{1}$ we obtain the equation $0=Pq+r.$
Applying the operator $Q\left(  D\right)  $ we see that
\begin{equation}
	Q\left(  D\right)  \left(  Pq\right)  =0. \label{uniqueness}
\end{equation}
Thus it is sufficient for uniqueness to show from (\ref{uniqueness}) that
$q=0.$

For simplicity in the exposition we indicate next a special case of Theorem \ref{main}
(corresponding to the parameter $\tau= 0$).  Only after we mention the definition of the apolar inner product,  will the full Theorem \ref{main} be
stated. Recall that a
polynomial $P\left(  z\right)  $ is \emph{homogeneous} of degree $m$ if
$P\left(  tz \right)  =t^{m}P\left(  z\right)  $ for all $t>0$ and for all
$z$ (in particular, the constant zero polynomial is homogeneous of every
degree).

\begin{theorem}
	\label{mainSp} Let $0 \le\beta\leq k-1 $, let $P$ be a non-homogeneous polynomial of degree $k
	\ge 2$ with homogenous expansion $P=P_{0}+\cdots+P_{\beta}+P_{k} $, and let
	$\varphi:\mathbb{C}^{d}\rightarrow\mathbb{C}$ be an entire function which
	satisfies $P_{k}^{\ast}\left(  D\right)  \left(  P\varphi\right)  =0$ and has
	order $\rho< 2 k^{-1}\left(  k-\beta\right)  $. Then $\varphi=0.$
\end{theorem}

In terms of the Fischer operator $F_{QP}, $ defined via $F_{QP} (\varphi) :=
Q(D) (P\varphi), $ the conclusion of Theorem \ref{mainSp} states that
$F_{P_{k}^{*}P}$ is injective on the entire functions of sufficiently low order.

\vskip .2 cm

Note that we only deal with the case of a non-homogeneous polynomial $P$. The bijectivity of the Fischer operator on the space of {\em all entire functions} when $P$ is homogeneous, was proved by H. S. Shapiro in \cite[Theorem 1]{Shap89}.

\vskip .2 cm

We speak of \emph{weak Fischer pairs} when the decomposition $f=P\cdot q+r$ is
not required to be unique. In reference \cite[Theorem 2]{AlRe23}, the authors
proved the existence of weak Fischer decompositions, under the bound
$\rho(k-\tau) < 2(k-\beta) $ on the order $\rho$ of $f$, where the parameter $\tau$ quantifies the strength of the \emph{Khavinson-Shapiro bounds}, cf. the
inequalities (\ref{cond3}) below. Our main theorem partially complements this existence
result, yielding
the injectivity of the Fischer operator under 
sometimes more strict assumptions on $\rho$, depending on the parameters.

We are indebted to two anonymous referees for their careful reading of this paper and several usefull suggestions.

\section{A lemma on sequences of nonnegative numbers}

\begin{lemma}
	\label{lemmaSeq} Let $\left\{  a_{m}\right\}  _{m\in\mathbb{N}}$ be a sequence
	of non-negative numbers. Suppose there is a finite non-empty subset
	$E\subset\mathbb{N}\setminus\{0\}$ satisfying the following two conditions:
	
	(i) There exist constants $A \ge 1$,   $D \ge 0$ and $\alpha\in\mathbb{R}$ such that for
	all $m\in\mathbb{N}$,
	\[
	a_{m}\leq A\left(  m+D\right)  ^{\alpha}\max_{j\in E}a_{m+j}.
	\]
	(ii) There exist constants $A_{0},b_{0}>0$,  $D_{0},\alpha_{0}\geq0$ and
	$\sigma\neq0$ such that for all $m\in\mathbb{N}\setminus\{0\}$,
	\[
	a_{m}\leq A_{0}\ \frac{\left(  m+D_{0}\right)  ^{\alpha_{0}}}{b_{0}^{m}%
	}\ m^{-\frac{m}{\sigma}}\text{ }.
	\]
	Let $\beta_{\ast}:=\min E$ and $\beta^{\ast}=\max E.$ If
	\[
	0\leq\alpha<\frac{\beta_{\ast}}{\sigma},\text{ \quad or }\quad\alpha<0\text{
		and }\alpha<\frac{\beta^{\ast}}{\sigma},
	\]
	then for all $m\in\mathbb{N}$ we have $a_{m}=0$. When $\alpha
	=\frac{\beta_{\ast}}{\sigma} > 0$, the same conclusion holds under the
	additional hypothesis that $A < b_{0}^{\beta_{\ast}}.$
\end{lemma}

\begin{proof}
	Assume that $a_{m} > 0$ for some $m \ge 0$. We derive a contradiction. If $m =0$, by condition $(i)$ there exists a $j \in E$ such that $a_{m + j} > 0$, so we may suppose that $m > 0$. Since
	\begin{equation}
		a_{m}\leq A\left(  m+D\right)  ^{\alpha}\max_{j\in E}a_{m+j},\label{eq1}%
	\end{equation}
	there exists an $l_{1}\in E$ with
	\[
	a_{m}\leq A\left(  m+D\right)  ^{\alpha}a_{m+l_{1}}.
	\]
	Applying inequality (\ref{eq1}) to $a_{m+l_{1}}$ instead of $a_{m}$, we
	see that there exists an $l_{2}\in E$ such that
	\[
	a_{m}\leq A^{2}\left(  m+D\right)  ^{\alpha}\left(  m+l_{1}+D\right)
	^{\alpha}a_{m+l_{1}+l_{2}}.
	\]
	Now proceed inductively, to conclude that there exist $l_{1}, \dots,l_{j}\in E$
	satisfying
	\[
	a_{m}\leq A^{j}\left(  m+D\right)  ^{\alpha}\left(  m+l_{1}+D\right)
	^{\alpha}\cdots\left(  m+D+l_{1}+\cdots+l_{j-1}\right)  ^{\alpha}%
	a_{m+l_{1}+\cdots+l_{j}}.
	\]
	Define for $j \ge 1$, 
	\[
	p_{m,j} :=\left(  m+D\right)  \left(  m+l_{1}+D\right)  \cdots\left(
	m+D+l_{1}+\cdots+l_{j-1}\right),
	\]
	and $k_{j}:=
	m+l_{1}+\cdots+l_{j}.$ Our assumption $a_{m}\leq A_{0}\frac{\left(
		m+D_{0}\right)  ^{\alpha_{0}}}{b_{0}^{m}}\frac{1}{m^{\frac{m}{\sigma}}}$
	implies that
	\[
	a_{k_{j}}\leq A_{0}\frac{\left(  k_{j}+D_{0}\right)  ^{\alpha_{0}}}%
	{b_{0}^{k_{j}}}k_{j}^{-\frac{k_{j}}{\sigma}}, \
	\]
	so we have
	\[
	a_{m}\leq A_{0}\frac{\left(  k_{j}+D_{0}\right)  ^{\alpha_{0}}}{b_{0}^{k_{j}}%
	}A^{j}p_{m,j}^{\alpha}k_{j}^{-\frac{k_{j}}{\sigma}}.
	\]
	Taking $k_{j}$-th roots in the last inequality shows that
	\begin{equation}
		a_{m}^{\frac{1}{k_{j}}}\leq A_{0}^{\frac{1}{k_{j}}}\frac{\left(  k_{j}%
			+D_{0}\right)  ^{\frac{\alpha_{0}}{k_{j}}}}{b_{0}}A^{\frac{j}{k_{j}}}%
		p_{m,j}^{\frac{\alpha}{k_{j}}}k_{j}^{-\frac{1}{\sigma}}.\label{eqNEW}%
	\end{equation}
	Since $k_{j}=m+l_{1}+\cdots+l_{j}\geq m+\beta_{\ast}j > \beta_{\ast}j \ge j$,   we see
	that as $j \to \infty$, $k_{j}\rightarrow\infty$, so
	$a_{m}^{\frac{1}{k_{j}}}\rightarrow 1$,  $A_{0}
	^{\frac{1}{k_{j}}}\rightarrow 1$, $\left(  k_{j}+D_{0}\right)
	^{\frac{\alpha_{0}}{k_{j}}}\rightarrow1$, and $A^{\frac{j}{k_{j}}} \le
	A^{\frac{1}{\beta_{\ast}}}$. Thus, it suffices  to prove
	that $p_{m,j}^{\frac{\alpha}{k_{j}}}k_{j}^{-\frac{1}{\sigma}}\rightarrow0$ in order to
	obtain a contradiction, since then the left hand side of inequality (\ref{eqNEW}) converges to
	$1$ and the right hand side to $0.$ 
	
	If $
	0\leq\alpha
	<
	\beta_{\ast}/\sigma$, then  $m+D+l_{1}+\cdots+l_{s}\leq D+k_{j}$ for
	$s=0, \dots ,j-1$ and $p_{m,j}\leq\left(  D+k_{j}\right)  ^{j},$ so
	\[
	p_{m,j}^{\frac{\alpha}{k_{j}}}k_{j}^{-\frac{1}{\sigma}}
	\leq
	\left(  D+k_{j}\right)
	^{\alpha\frac{j}{k_{j}}}k_{j}^{-\frac{1}{\sigma}}\leq\left(  D+k_{j}\right)
	^{\frac{\alpha}{\beta_{\ast}}}k_{j}^{-\frac{1}{\sigma}}=\left(  \frac{D+k_{j}
	}{k_{j}}\right)  ^{\frac{\alpha}{\beta_{\ast}}}k_{j}^{\left(  \frac{\alpha
		}{\beta_{\ast}}-\frac{1}{\sigma}\right)  }.
	\]
	Since $\frac{\alpha}{\beta_{\ast}}-\frac{1}{\sigma}<0$ the term on the right
	hand side converges to $0$. If  $
	0 <\alpha
	=
	\beta_{\ast}/\sigma$, then
	\[
	1=\lim\sup_{j\rightarrow\infty}a_{m}^{\frac{1}{k_{j}}}
	\leq
	\lim\sup
	_{j\rightarrow\infty}A_{0}^{\frac{1}{k_{j}}}\frac{\left(  k_{j}+D_{0}\right)
		^{\frac{\alpha_{0}}{k_{j}}}}{b_{0}}A^{\frac{j}{k_{j}}}\left(  \frac{D+k_{j}
	}{k_{j}}\right)  ^{\frac{\alpha}{\beta_{\ast}}}\leq\frac{1}{b_{0}}A^{\frac
		{1}{\beta_{\ast}}} < 1.
	\] 
	This finishes the proof when $\alpha\geq0.$
	
	Let us discuss next the case $\alpha<0.$ Since $k_{j}\geq\beta_{\ast}j\ $for
	$j\geq1$, we have
	\[
	p_{m,j}
	=
	\left(  m+D\right)  \left(  m+l_{1}+D\right)  \cdots\left(
	m+D+l_{1}+\cdots+l_{j-1}\right)  
	\geq
	\beta_{\ast}^{j-1}\left(  j-1\right)  !.
	\]
	From $\alpha<0$ we conclude that
	\[
	p_{m,j}^{\frac{\alpha}{k_{j}}}\leq\left(  \beta_{\ast}\right)^{\alpha
		\frac{j-1}{k_{j}}}\left[  \left(  j-1\right)  !\right]  ^{\frac{\alpha}{k_{j}%
	}}.
	\]
	Recall that  $\beta^{\ast} :=\max\left\{  l:l\in E\right\}.$ Since
	$k_{j}=m+l_{1}+\cdots+l_{j}$, it follows that
	\[
	\beta_{\ast}j\leq k_{j}\leq m+\beta^{\ast}j.
	\]
	Thus $\left(  \beta_{\ast}\right)  ^{\alpha\frac{j-1}{k_{j}}}$ is bounded above.
	Furthermore, a version of Stirling's formula 
	with explicit bounds for $m \ge 0$ tells us that
	\[
	m!\geq\left(  2\pi m\right)  ^{\frac{1}{2}}\left(  \frac{m}{e}\right)
	^{m}e^{\frac{1}{12m+1}}.
	\]
	Using $\alpha<0$ and the preceding formula with $m=j-1 \ge 1$ we obtain
	\[
	\left[  \left(  j-1\right)  !\right]  ^{\frac{\alpha}{k_{j}}}\leq\left(
	2\pi\left(  j-1\right)  \right)  ^{\frac{1}{2}\frac{\alpha}{k_{j}}}\left(
	\frac{j-1}{e}\right)  ^{\alpha\frac{j-1}{k_{j}}}e^{\frac{\alpha}{k_j (12\left(
			j-1\right)  +1)}}.
	\]
	Note that
	\[
	\left(  \frac{1}{e}\right)  ^{\alpha\frac{j-1}{k_{j}}}=e^{-\alpha\frac
		{j-1}{k_{j}}}\leq e^{-\frac{\alpha}{\beta_{\ast}}},
	\]
	so
	\[
	\lim\sup_{j\rightarrow\infty}p_{m,j}^{\frac{\alpha}{k_{j}}}k_{j}^{-\frac{1}{\sigma}
	}
	\leq
	\lim\sup_{j\rightarrow\infty}  \left(  \beta_{\ast}\right)^{\alpha
		\frac{j-1}{k_{j}}} e^{-\frac{\alpha}{\beta_{\ast}}}\left(
	j-1\right)  ^{\alpha\frac{j-1}{k_{j}}}k_{j}^{-\frac{1}{\sigma}}
	\]
	Since $k_{j}\leq m+\beta^{\ast}j$ we have
	\[
	\frac{k_{j}}{j-1}
	\leq
	\frac{m}
	{j-1}+\beta^{\ast}+\frac{1}{j - 1}\beta^{\ast}.
	\]
	By assumption $\alpha  < \beta^{\ast}/  \sigma$, where $\alpha < 0$ (and we can have either $\sigma < 0$ or $\sigma > 0$). Let $\varepsilon>0$ be such that $\alpha <  (1 + \varepsilon)\beta^{\ast}/\sigma$. Now there exists a $j_{\varepsilon}$ such that for all
	$j\geq j_{\varepsilon}$,
	\[
	\frac{k_{j}}{j-1}\leq\left(  1+\varepsilon\right)  \beta^{\ast}.
	\]
	Then whenever $j\geq j_{\varepsilon}$,
	\[
	j-1
	\geq
	k_{j}\left(
	1+\varepsilon\right)^{-1}\left(  \beta^{\ast}\right)^{-1}.
	\]
	Using again that $\alpha<0$ we get
	\begin{align*}
		\left(  j-1\right)  ^{\alpha\frac{j-1}{k_{j}}} &  \leq k_{j}^{\alpha\frac
			{j-1}{k_{j}}}\left(  1+\varepsilon\right)  ^{-\alpha\frac{j-1}{k_{j}}}\left(
		\beta^{\ast}\right)  ^{-\alpha\frac{j-1}{k_{j}}}\\
		&  \leq k_{j}^{\alpha\frac{j-1}{k_{j}}}\left(  1+\varepsilon\right)
		^{-\alpha\frac{1}{\beta_{\ast}}}\left(  \beta^{\ast}\right)  ^{-\alpha\frac
			{1}{\beta_{\ast}}}.
	\end{align*}
	Now
	\[
	\lim\sup_{j\rightarrow\infty}p_{m,j}^{\frac{\alpha}{k_{j}}}k_{j}^{-\frac{1}{\sigma}
	}
	\leq 
	e^{-\frac{\alpha}{\beta_{\ast}}}\left(  1+\varepsilon\right)
	^{-\alpha\frac{1}{\beta_{\ast}}}\left(  \beta^{\ast}\right)  ^{-\alpha\frac
		{1}{\beta_{\ast}}}
	\lim\sup_{j\rightarrow\infty}
	\left(  \beta_{\ast}\right)^{\alpha
		\frac{j-1}{k_{j}}}
	k_{j}^{\alpha\frac{j-1}{k_{j}}
	}k_{j}^{-\frac{1}{\sigma}},
	\]
	and this limit equals $0$ since
	\[
	\alpha\frac{j-1}{k_{j}}
	\leq
	\frac{\alpha}{\left(
		1+\varepsilon\right)  \beta^{\ast}}
	< \frac{1}{\sigma}.
	\]
	\end{proof}

\section{The apolar inner product, and statement of the main result}

Let
\[
P\left(  z\right)  =\sum_{\alpha\in\mathbb{N}^{d},\left\vert
\alpha\right\vert \leq N}c_{\alpha}z^{\alpha}\text{ and }Q\left(  z\right)
=\sum_{\alpha\in\mathbb{N}^{d},\left\vert \alpha\right\vert \leq
M}d_{\alpha}z^{\alpha}
\] be complex polynomials in $d$ variables.
The \emph{apolar inner product} $\langle\cdot,\cdot\rangle_{a}$ on
$\mathbb{C}[z_{1},\dots,z_{d}]$ is defined by
\begin{equation}
	\label{apolar}\left\langle P,Q\right\rangle _{a} :=\left(  Q^{*}\left(
	D\right)  P\right)  \;(0)=\sum_{\alpha\in\mathbb{N}^{d}}\alpha!c_{\alpha
	}\overline{d_{\alpha}}.
\end{equation}
Note that for monomials $z^{\alpha}$ and $z^{\beta}$, definition
(\ref{apolar}) means that $\left\langle z^{\alpha}, z^{\beta}\right\rangle
_{a} = 0$ if $\alpha\ne\beta$, and $\left\langle z^{\alpha}, z^{\alpha
}\right\rangle _{a} = \alpha!$. Furthermore, from these relations definition
(\ref{apolar}) is recovered by sesquilinearity.

A fundamental property of the apolar inner product, easily obtained from the
definition (alternatively, see \cite[Proposition 3]{AlRe23}) is that given
polynomials $Q, f$ and $g$, we have
\begin{equation}
	\left\langle Q^{\ast}\left(  D\right)  f,g\right\rangle _{a}=\left\langle
	f,Q\cdot g\right\rangle _{a} . \label{eqFischeradjoint}%
\end{equation}
From the apolar inner product we obtain the associated \emph{apolar norm} in
the usual way: $\left\Vert f\right\Vert _{a} := \sqrt{\left\langle
	f,f\right\rangle _{a}}. $

\vskip .2 cm We are now ready to state our main result. Regarding the notation used below on the non-homogeneous polynomial
$P= P_{\beta_{1}}+\cdots+P_{\beta_{2}}+P_{k}$, where the degrees of the homogeneous parts increase from left to right,
we suppose not only that $P_k \ne 0$, but also that  $P_{\beta_{1}} \ne 0$ and 
$P_{\beta_{2}} \ne 0$; however, it may happen 
that $
\beta_{1} = \beta_{2}$, so actually $P= P_{\beta_{1}} +P_{k}$.
\begin{theorem}
	\label{main} Let $0\leq\beta_{1}\leq\beta_{2}\leq k-1$ and let $P$ be a
	non-homogeneous polynomial of degree $k\geq 2$ with homogenous expansion $P= P_{\beta_{1}}+\cdots+P_{\beta_{2}}+P_{k}$,  where  $P_{\beta_{1}} \ne 0$ and 
	$P_{\beta_{2}} \ne 0$.
	Suppose there exist constants $\tau\geq0$ with $\tau\leq k -1$, and $C>0$, such
	that for all natural numbers $m\geq0$ and all homogeneous polynomials $g_{m}$
	of degree $m$, we have
	\begin{equation}
		\left\Vert P_{k}g_{m}\right\Vert _{a}\geq C\left(  m+1\right)  ^{\tau
			/2}\left\Vert g_{m}\right\Vert _{a}.\label{cond3}
	\end{equation}
	Let $
	\rho <  2\left(  k-\beta_{2}\right) / (k-\tau)$   if $\beta_{2}
	-\tau\geq0$, and let $\rho < 2(k-\beta_{1}) / (k+\beta_{2}-\beta_{1}-\tau)$  when $\beta_{2}
	-\tau < 0$.
	If $\varphi:\mathbb{C}^{d}\rightarrow\mathbb{C}$ is an entire function of
	order $\rho$ satisfying $P_{k}^{\ast}\left(  D\right)  \left(  P\varphi
	\right)  =0$, then $\varphi=0.$
\end{theorem}

Note that (\ref{cond3}) always holds when $\tau=0$, as a consequence of
Bombieri's inequality, cf. \cite{BBEM90}. When $\tau\ge1$, we call the
inequalities (\ref{cond3}) \emph{Khavinson-Shapiro bounds}, cf. \cite{AlRe23}.
It is shown in \cite[Theorem 5]{AlRe23} that when the number of variables $d > 1$, and $k > 1$, a necessary contition for  (\ref{cond3}) to hold is having $\tau \le  k - 1$. 
A more detailed discussion of condition (\ref{cond3}) can be found in
\cite{AlRe24}. If $k = 1$, the Fischer operator $F_{P_{k}^{*}P}$ is always bijective on the space of entire functions, by \cite[Theorem 8]{AlRe23}. Regarding the bounds on the order, we mention that if $\beta_{2}
-\tau\geq0$, we have
$$
\frac{2}{ k}
\le
\frac{2 (k - \beta_{2})}{ k-\tau}
\le
2,
$$
while  if $\beta_{2}
-\tau <0$, then 
$$
2
<
\frac{2 (k - \beta_{1})}{k+\beta_{2}-\beta_{1}-\tau}
\le
2 k.
$$

With respect to the existence result of \cite[Theorem 2]{AlRe23}, there the assumption on the order is 
$$
\rho
<
\frac{2 (k - \beta_{2})}{ k-\tau},
$$
which coincides with the uniqueness bound for $\rho$ when
$\beta_{2} \ge \tau$. But  if $\beta_{2}
<\tau$, then 
$$
\frac{2 (k - \beta_{1})}{k+\beta_{2}-\beta_{1}-\tau}
\le
\frac{2 (k - \beta_{2})}{ k-\tau},
$$
so the restriction on the order is at least as strong for the proof of uniqueness as for existence.

\section{Entire functions of finite order}

For $z \in\mathbb{C}^{d}$ we write $\left\vert z\right\vert =\sqrt{\left\vert
	z_{1}\right\vert ^{2}+ \cdots+\left\vert z_{d}\right\vert ^{2}}.$  We use the standard notations for multi-indices $\alpha=\left(
\alpha_{1}, \dots,\alpha_{d}\right)  \in\mathbb{N}^{d}$, namely $\alpha
!=\alpha_{1}! \cdots\alpha_{d}!$, $\left\vert \alpha\right\vert =\alpha
_{1}+\cdots+\alpha_{d}$, and $z^{\alpha}= z_{1}^{\alpha_{1}} \cdots
z_{d}^{\alpha_{d}}.$ We also write $\mathbb{S} ^{2d-1}=\left\{  z\in
\mathbb{C}^{d}:\left\vert z\right\vert =1\right\}  .$

Suppose that $f:\mathbb{C} ^{d}\rightarrow\mathbb{C}$ is an entire function.
Let
\begin{equation}
	M_{\mathbb{C}^{d}}\left(  f,r\right)  :=\sup\left\{  \left\vert f\left(
	z\right)  \right\vert :z\in\mathbb{C}^{d},\left\vert z\right\vert =r\right\}
	.
\end{equation}
The order $\rho_{\mathbb{C}^{d}}\left(  f\right)  $ of $f$ is then defined as
\[
\rho_{\mathbb{C}^{d}}\left(  f\right)  =\lim_{r\rightarrow\infty}\sup\frac{
	\log\log M_{\mathbb{C}^{d}}\left(  f,r\right)  }{\log r}.
\]
It is easy to see that $f$ has finite order $\leq\rho$ if and only if for all
$\varepsilon>0$ there exists a positive number $r_{\varepsilon}$ such that
\[
M_{\mathbb{C}^{d}}\left(  f,r\right)  \leq e^{r^{\rho+\varepsilon}}\text{ for
	all }r\geq r_{\varepsilon}.
\]
Using the Taylor expansion one can write
\[
f\left(  z\right)  =\sum_{m=0}^{\infty}f_{m}\left(  z\right)  \text{ where }
f_{m}\left(  z\right)  =\sum_{\left\vert \alpha\right\vert =m}c_{m,\alpha
}z^{\alpha},
\]
so each $f_{m}\left(  z\right)  $ is a homogeneous polynomial of degree $m$.
Given $\rho\ge0$, it is well known that $\rho_{ \mathbb{C}^{d}}\left(
f\right)  \leq\rho$ if and only if for every $\varepsilon>0$, there exists an
$m_{0} \geq0$ such that for every $m \ge m_{0}$ the following bounds hold:
\begin{equation}
	\max_{\theta\in\mathbb{S}^{2 d-1}}\left\vert f_{m}\left(  \theta\right)
	\right\vert \leq\frac{1}{m^{m/\left(  \rho+\varepsilon\right)  }}.
	\label{eqtaylor}%
\end{equation}
As in the complex case, for $x=\left(  x_{1},\dots,x_{d}\right)  \in
\mathbb{R}^{d}$ the euclidean distance $\left\vert x\right\vert $ is given by
$\left\vert x\right\vert ^{2}=x_{1}^{2}+\cdots+x_{d}^{2}$. Let $B_{R}%
:=\left\{  x\in\mathbb{ R}^{n}:\left\vert x\right\vert <R\right\}  $ be the
open ball in $\mathbb{R}^{n}$ with center $0$ and radius $0<R\leq\infty.$

The following result is well known, see for instance \cite[Lemma 17]{AlRe23}.

\begin{lemma}
	\label{lemma17} \label{Bargma} Let $f_{m}$ be a homogeneous polynomial of
	degree $m$, and denote by $\mathbb{S}^{2d-1}$ the unit sphere in
	$\mathbb{R}^{2n}$. There is a dimensional constant $C_{d}>0$ such that,
	identifying $\mathbb{C}^{n}$ with $\mathbb{R}^{2n}$ as a measure space, we
	have
	\[
	\left\Vert f_{m}\right\Vert _{a}\leq C_{d}\sqrt{\left(  m+d-1\right)  !}
	\max_{\theta\in\mathbb{S}^{2d-1}}\left\vert f_{m}\left(  \theta\right)
	\right\vert .
	\]
	
\end{lemma}

\begin{proposition}
	\label{cond2} Assume that $f:\mathbb{C}^{d}\rightarrow\mathbb{C}$ is entire
	and has order $\leq\rho.$ Then for every $m\geq1$ and every $\varepsilon>0$,
	we have $\left\Vert f_{m}\right\Vert _{a} \leq2\sqrt{\pi} \ C_{d} \ e^{- m/2}
	\ \left(  m+d-1\right)  ^{d/2} m^{m\left(  \frac{1}{2} - \frac{1}%
		{\rho+\varepsilon} \right)  }. $
\end{proposition}

\begin{proof}
	By Lemma \ref{lemma17} we know that
	\[
	\left\Vert f_{m}\right\Vert _{a}\leq C_{d}\sqrt{\left(  m+d-1\right)  !}%
	\frac{1}{m^{m/\left(  \rho+\varepsilon\right)  }}.
	\]
	Furthermore $\left(  m+d-1\right)  !\leq m!\left(  m+d-1\right)  ^{d}$. A
	standard version of Stirling's formula with explicit bounds for every $m\ge1$,
	yields
	\[
	m! < \sqrt{2\pi m}\left(  \frac{m}{e}\right)  ^{m}e^{\frac{1}{12m}} < 2
	\sqrt{\pi m}\left(  \frac{m}{e}\right)  ^{m}%
	\]
	since $ e^{\frac{1}{12 m}} < \sqrt{2}$. It follows that
	\begin{align*}
		\left\Vert f_{m}\right\Vert _{a}  &  \leq2C_{d}\sqrt{\pi}\left(  m+d-1\right)
		^{d/2}\left(  \frac{m}{e}\right)  ^{m/2}\frac{1}{m^{m/\left(  \rho
				+\varepsilon^{}\right)  }}\\
		&  =2C_{d}\sqrt{\pi} \ \frac{\left(  m+d-1\right)  ^{d/2}}{e^{m/2} \ }\frac
		{1}{ m^{m\left(  \frac{1}{\rho+\varepsilon}-\frac{1}{2}\right)  }}.
	\end{align*}
	
\end{proof}

With different normalizations, the following result is essentially due to B.
Beauzamy, cf. \cite[Formula (6)]{Beau97}. The exact statement appearing below is
\cite[Lemma 13]{AlRe23}.

\begin{lemma}
	\label{lemmaD} If $P\left(  z \right)  = {\displaystyle\sum_{\left\vert
			a\right\vert =k}} c_{\alpha}z^{\alpha}$ is a homogeneous polynomial of degree
	$k$, then
	\[
	\left\Vert Pf_{m}\right\Vert _{a}\leq\left\Vert f_{m}\right\Vert _{a}\left(
	1+m\right)  ^{\frac{k}{2}} {\displaystyle\sum_{\left\vert \alpha\right\vert
			=k}} \left\vert c_{\alpha}\right\vert \sqrt{\alpha!}.
	\]
	
\end{lemma}

\section{Proof of Uniqueness}

We are now ready to prove Theorem \ref{main}.

\begin{proof}
	Let $\varphi:\mathbb{C}^{d}\rightarrow\mathbb{C}$ be an entire function of
	order $\rho<\frac{2\left(  k-\beta_{2}\right)  }{k-\tau}$, and let $P$ be a
	polynomial of degree $k\geq1$ with homogenous expansion $P=P_{\beta_{1}%
	}+\cdots+P_{\beta_{2}}+P_{k}$, where $0\leq\beta_{1}\leq\beta_{2}\leq k-1.$
	Let us write $\varphi=\sum_{m=0}^{\infty}\varphi_{m}$, with each $\varphi_{m}$
	a homogeneous polynomial of degree $m$. Then $u:=P\varphi$ can be expressed as
	$u=\sum_{m=0}^{\infty}u_{m},$ where the homogeneous polynomials $u_{m}$ of
	degree $m$ satisfy
	\begin{equation}
		u_{m+k}=P_{k}\varphi_{m}+....+P_{0}\varphi_{m+k}.\label{eqrecurs}%
	\end{equation}

	Note that $P_{k}^{\ast}\left(  D\right)  u=0$ implies that $P_{k}^{\ast
	}\left(  D\right)  u_{m+k}=0$ for every $m \ge 0$, by uniqueness of the representation of $u$ as a
	series of homogeneous polynomials. Taking the inner product with $P_{k}%
	\varphi_{m}$ on both sides of (\ref{eqrecurs}) and using $P_{k}^{\ast}\left(
	D\right)  u_{m+k}=0$, we obtain
	\[
	\left\Vert P_{k}\varphi_{m}\right\Vert _{a}^{2}+\sum_{j=1}^{k}\left\langle
	P_{k-j}\varphi_{m+j},P_{k}\varphi_{m}\right\rangle _{a}=\left\langle
	u_{m+k},P_{k}\varphi_{m}\right\rangle _{a}=\left\langle P_{k}^{\ast}\left(
	D\right)  u_{m+k},\varphi_{m}\right\rangle _{a}=0,
	\]
	so
	\[
	\left\Vert P_{k}\varphi_{m}\right\Vert _{a}^{2}\leq\sum_{j=1}^{k}\left\vert
	\left\langle P_{k-j}\varphi_{m+j},P_{k}\varphi_{m}\right\rangle _{a}%
	\right\vert \leq\sum_{j=1}^{k}\left\Vert P_{k-j}\varphi_{m+j}\right\Vert
	_{a}\left\Vert P_{k}\varphi_{m}\right\Vert _{a}%
	\]
	by the Cauchy-Schwarz inequality. Thus $\left\Vert P_{k}\varphi_{m}\right\Vert
	_{a}\leq\sum_{j=1}^{k}\left\Vert P_{k-j}\varphi_{m+j}\right\Vert _{a}.$ By
	Lemma \ref{lemmaD}, for each $j\in\{1,\dots,k\}$ there exists a constant
	$D_{k-j}$, which depends only on $P_{k-j},$ and satisfies
	\[
	\left\Vert P_{k-j}\varphi_{m+j}\right\Vert _{a}\leq D_{k-j}\left\Vert
	\varphi_{m+j}\right\Vert _{a}\left(  1+m\right)  ^{\left(  k-j\right)  /2}.
	\]
	Defining
	\[
	E:=\left\{  j\in\left\{  1,\dots,k\right\}  :P_{k-j}\neq0\right\}
	\subset\left\{  k-\beta_{2},\dots,k-\beta_{1}\right\}  ,
	\]
	we conclude that
	\begin{align*}
		\left\Vert P_{k}\varphi_{m}\right\Vert _{a} &  \leq\sum_{j\in E}%
		D_{k-j}\left\Vert \varphi_{m+j}\right\Vert _{a}\left(  1+m\right)  ^{\left(
			k-j\right)  /2}\\
		&  \leq\max_{s\in E}\left\Vert \varphi_{m+s}\right\Vert _{a}\sum_{j\in
			E}D_{k-j}\left(  1+m\right)  ^{\left(  k-j\right)  /2}.
	\end{align*}
	Note also that from the assumptions
	$P_{\beta_{1}} \ne 0$ and 
	$P_{\beta_{2}} \ne 0$ we obtain $\beta_* = k - \beta_2$ and  $\beta^* = k - \beta_1$.
	The change of indices $n:=k-j$ leads to
	\begin{align*}
		\sum_{j\in E}D_{k-j}\left(  1+m\right)  ^{\left(  k-j\right)  /2} &
		=\sum_{n= \beta_1 }^{\beta_{2}}D_{n}\left(  1+m\right)  ^{n/2}\\
		&  = \left(  1+m\right)  ^{\beta_{2}/2}\left(  \sum_{n= \beta_1 }^{\beta_{2}}
		\frac{D_{n}}{\left(  1+m\right)^{\left(  \beta_{2}-n\right) /2}}\right).
	\end{align*}
	Next we fix $0<\delta\ll1$ and take $m_{0}=m_{0}(\delta)\gg1$ so that
	\[
	\sum_{n = \beta_1 }^{\beta_{2}}\frac{D_{n}}{\left(  1+m_{0}\right)^{\left(  \beta
			_{2}- n \right)  /2}}\leq D_{\beta_{2}}+\delta.
	\]
	By hypothesis there exist $\tau\geq0$ and $C>0$ such that for all $m \ge m_0$,
	\[
	C\left(  m+1\right)  ^{\tau/2}\left\Vert \varphi_{m}\right\Vert_a \leq\left\Vert
	P_{k}\varphi_{m}\right\Vert _{a}.
	\]
	Hence
	\[
	\left\Vert \varphi_{m}\right\Vert _{a}\leq C^{-1}\left(  D_{\beta_{2}}%
	+\delta\right)  \frac{\left(  1+m\right)  ^{\beta_{2}/2}}{\left(  1+m\right)
		^{\tau/2}}\max_{j\in E}\left\Vert \varphi_{m+j}\right\Vert _{a}\ .
	\]
	Let us write $a_{m}:=\left\Vert \varphi_{m}\right\Vert _{a}$, $\alpha:=\left(
	\beta_{2}-\tau\right)  /2$,  and $A:= \max\{1,  C^{-1}\left(  D_{\beta_{2}}%
	+\delta\right)\}  $. Then
	\[
	a_{m}\leq A\left(  m+1\right)  ^{\alpha}\max_{j\in E}a_{m+j}%
	\]
	for all $m\in\mathbb{N}$, so the first condition in Lemma \ref{lemmaSeq} is
	satisfied. To check the second condition we use Proposition \ref{cond2}, which
	tells us that for every $\varepsilon>0$ the following inequality holds:
	\[
	a_{m}\leq2\sqrt{\pi}C_{d}\frac{\left(  m+d-1\right)  ^{d/2}}{e^{m/2}}\frac
	{1}{m^{m\left(  \frac{1}{\rho+\varepsilon}-\frac{1}{2}\right)  }}.
	\]
	so it holds for $\sigma=\sigma_{\varepsilon}:= (\frac{1}{\rho+\varepsilon}
	-\frac{1}{2})^{-1}.$  
	
	Assume that $\beta_{2}-\tau\geq0,$ so $\alpha\geq0$ and $\rho<2\frac{k-\beta_2
	}{k-\tau}.$ Choose $\varepsilon>0$ so that
	\begin{equation}\label{ep}
		\rho+\varepsilon<2\frac{k-\beta_2}{k-\tau}
		\le 2.
	\end{equation}
	Since $\frac{1}{\sigma}=\frac{1}{\rho+\varepsilon}-\frac{1}{2}$, 
	by Lemma \ref{lemmaSeq} we have that $a_{m} = 0$ for every $m \ge 0$, provided $\frac{1}%
	{\sigma}>\frac{\alpha}{\beta_{\ast}}.$ This condition means that
	\[
	\frac{1}{\rho+\varepsilon}-\frac{1}{2}
	>
	\frac{1}{2}\frac{\beta_2-\tau}{k-\beta_2},
	\] 
	which is equivalent to the first inequality in (\ref{ep}).
	
	Now assume that $\beta_{2}-\tau<0$, and thus $\alpha < 0$. By the bound on $\rho$  there exists an $\varepsilon > 0$ such that 
	\[
	2 
	\ne
	\rho + \varepsilon
	<
	2\frac{k-\beta_{1}}{k+\beta_{2}-\beta_{1}-\tau}.
	\]
	Setting
	$\sigma:=\frac{1}{\rho+\varepsilon}%
	-\frac{1}{2}$, we conclude that 
	\[
	\frac{\alpha}{\beta^{\ast}}
	=
	\frac{1}{2}\frac{\left(  \beta_{2}-\tau\right)  }{k-\beta_{1}}<\frac{1}%
	{\sigma},
	\]
	so once again the hypotheses of Lemma \ref{lemmaSeq} are satisfied.
\end{proof}

\end{document}